\documentclass[10pt,leqno]{amsart}%
\usepackage{amsmath}
\usepackage{amsfonts}
\usepackage{amssymb}
\usepackage{graphicx}%
\newtheorem{Thm}{Theorem}[section]
\newtheorem{Lem}[Thm]{Lemma}

\newtheorem{corollary}[Thm]{Corollary}
\theoremstyle{definition}

\newtheorem{remark}[Thm]{Remark}

\numberwithin{equation}{section}

\newcommand{\define}[1]{{\em #1\/}}

\newcommand{\vp}{\varphi}

\newcommand{\cA}{{\mathcal A}}
\newcommand{\cB}{{\mathcal B}}

\newcommand{\field}[1]{\mathbb{#1}}
\newcommand{\N}{\field{N}}          		
\newcommand{\R}{\field{R}}          		
\newcommand{\Sph}{\field{S}}        		
\newcommand{\C}{\field{C}}          		
\newcommand{\hyp}{\field{H}}        		

\newcommand{\abs}[1]{\lvert #1 \rvert}

\newcommand{\loc}{{\rm loc}}

\DeclareMathOperator{\dist}{dist}

\DeclareMathOperator\diver{div}
\DeclareMathOperator\Hess{Hess}

\DeclareMathOperator\spt{spt}

\DeclareMathOperator\sect{Sect}

\DeclareMathOperator\Int{Int}

\begin{document}

\title[Nonsolvability of the asymptotic Dirichlet problem]{Nonsolvability of the asymptotic Dirichlet problems for some quasilinear elliptic PDEs  on Hadamard manifolds}

\author[Ilkka Holopainen]{Ilkka Holopainen}
\address{Department of Mathematics and Statistics, P.O. Box 68,
00014 University of Helsinki, Finland.}
\email{ilkka.holopainen@helsinki.fi}
\author[Jaime B. Ripoll]{Jaime B. Ripoll}
\address{UFRGS, Instituto de Matem\'atica, Av. Bento Goncalves 9500, 91540-000 Porto Alegre-RS, Brasil.}
\email{jaime.ripoll@ufrgs.br}
\thanks{I.H. supported by the Academy of Finland, project 252293; J.R. supported by the CNPq (Brazil) project 302955/2011-9.}



\begin{abstract}
We show, by modifying Borb\'ely's example, that there are $3$-dimen\-sional Cartan-Hadamard manifolds $M$, with 
sectional curvatures $\le -1$, such that the asymptotic Dirichlet problem for a class of quasilinear elliptic PDEs, including the minimal graph equation, is not solvable.
\end{abstract}

\maketitle

\bibliographystyle{acm}


\section{Introduction}\label{sec:intro}
In this paper we construct a $3$-dimensional Cartan-Hadamard manifold of sectional curvatures $\le -1$ where the asymptotic Dirichlet problem is not solvable with any continuous non-constant boundary data for a large class of equations
\begin{equation}\label{Q_equ}
Q[u]:= \diver\cA(\abs{\nabla u}^{2})\nabla u
\end{equation}
including, in particular, the minimal graph equation
\begin{equation}\label{min_eqn}
\diver\dfrac{\nabla u}{\sqrt{1+\abs{\nabla u}^{2}}}=0.
\end{equation}
Examples of such manifolds were earlier constructed for the usual Laplace equation by Ancona \cite{ancrevista} and Borb\'ely
\cite{Bor} and for the $p$-Laplace equation by Holopainen \cite{H_ns}, whereas no counterexamples, with $\sect \le -1$, were known for the minimal graph equation \eqref{min_eqn}.
Recall that a Cartan-Hadamard manifold $M$ is
a complete, connected and simply connected Riemannian $n$-manifold, $n\ge 2$,
of non-positive sectional curvature. By the Cartan-Hadamard theorem, the
exponential map $\exp_{o}\colon T_{o}M\to M$ is a diffeomorphism for
every point $o\in M$. In particular, $M$ is diffeomorphic to $\R^n$.
It is well-known that $M$ can be compactified by adding a natural geometric boundary, called the
\define{sphere at infinity} (or the \define{boundary at infinity}) and denoted
by $M(\infty)$, so that the resulting space $\bar M=M\cup M(\infty)$ equipped with
the \define{cone topology} will be homeomorphic to a closed Euclidean ball; see \cite{EO}.
The \define{Dirichlet problem at infinity} (or the \define{asymptotic Dirichlet problem}) on a Cartan-Hadamard manifold $M$ 
for the operator $Q$ is then the following: Given a continuous function $h$ on $M(\infty)$ does there exist a (unique) 
function $u\in C(\bar M)$ such that $Q[u]=0$ on $M$ and $u\vert M(\infty)=h$?

We assume that $\cA\colon (0,\infty)\to [0,\infty)$ is a smooth function such that 
\begin{equation}\label{Agrowth}
\cA(t)\le A_{0}t^{(p-2)/2}
\end{equation}
for all $t>0$, with some constants $A_{0}>0$ and $p\ge 1$, and that $\cB:=\cA'/\cA$ satisfies
\begin{equation}\label{Bgrowth}
-\frac{1}{2t} < \cB(t)\le \frac{B_{0}}{t}
\end{equation}
for all $t>0$ with some constant $B_{0}>-1/2$. Furthermore, we assume that 
\begin{equation}\label{A(0)}
t\cA(t^2)\to 0\text{ as }t\to 0+,
\end{equation}
and therefore we interprete $\cA(\abs{X}^2)X$ as a zero vector whenever $X$ is a zero vector.  
The equation $Q[u]=0$ is interpreted in weak sense. More precisely, a function $u$ is a solution to the equation
$Q[u]=0$ in an open set $\Omega\subset M$ if it belongs to the local Sobolev space
$W^{1,p}_{\loc}(\Omega)$ and
\begin{equation}\label{Qsol}
\int_{\Omega}\bigl\langle \cA\bigl(\abs{\nabla u}^{2}\bigr)\nabla u,\nabla\varphi\bigr\rangle dm =0
\end{equation}
for every $\varphi\in C^{\infty}_{0}(\Omega)$. Such function $u$ will be called a \define{$Q$-solution} in $\Omega$. 
Furthermore, we say that a function $u\in W^{1,p}_{\loc}(\Omega)$ is a \define{$Q$-subsolution} in $\Omega$ if 
$Q[u]\ge 0$ weakly in $\Omega$, that is 
\begin{equation}\label{Qsub}
\int_{\Omega}\bigl\langle\cA\bigl(\abs{\nabla u}^{2}\bigr)\nabla u,
\nabla\vp\bigr\rangle\,dm\le 0
\end{equation}
for every non-negative $\vp\in C^{\infty}_{0}(U)$. Similarly, a function $v\in W^{1,p}_{\loc}(U)$ is called a 
\define{$Q$-supersolution} in $\Omega$ if $-v$ is a $Q$-subsolution in $\Omega$.
Note that $u+c$ is a $Q$-solution (respectively, $Q$-subsolution, $Q$-supersolution) for every constant $c\in\R$ if 
$u$ is a $Q$-solution (respectively, $Q$-subsolution, $Q$-supersolution). Furthermore, $u$ and $-u$ are $Q$-solutions 
simultaneously. It follows from the growth condition \eqref{Agrowth} that test functions $\varphi$ in \eqref{Qsol} and 
\eqref{Qsub} can be taken from the class $W^{1,p}_{0}(\Omega)$ if $\abs{\nabla u}\in L^{p}(\Omega)$.

We call a relatively compact open set $\Omega\Subset M$ \define{$Q$-regular} if for any continuous boundary data 
$h\in C(\partial\Omega)$ there exists a unique $u\in C(\bar\Omega)$ which is a $Q$-solution in $\Omega$ and 
$u\vert \partial\Omega=h$. 
In addition to the growth conditions on $\cA$, we occasionally assume that
\begin{itemize}
\item[(A)] there is an exhaustion of $M$ by an increasing sequence of $Q$-regular domains 
$\Omega_{k}$, and that
\item[(B)] locally uniformly bounded sequences of continuous $Q$-solutions are compact in relatively compact subsets of $M$.
\end{itemize}

We want to emphasize that in this paper we do not study which operators satisfy the assumptions (A) and (B) above because
our primary motivation is a nonsolvability result for the asymptotic Dirichlet problem for the minimal graph equation
\eqref{min_eqn} which is known to fulfil all the conditions above. Indeed, 
\[
\cA(t)=\frac{1}{\sqrt{1+t}}\text{ and }
\cB(t)= -\frac{1}{2(1+t)}
\]
satisfy \eqref{A(0)} and growth conditions \eqref{Agrowth} and \eqref{Bgrowth} with $A_{0}=1$ and $B_{0}=0$, respectively. 
Furthermore, the condition (A) for the minimal graph equation follows from \cite[Theorem 2]{DLR} where the sets $\Omega_{k}$ can be chosen as 
geodesic balls $B(o,k)$ centered at a fixed point $o\in M$, and the condition (B) follows 
from \cite[Theorem 1.1]{Spruck} (see also \cite[Theorem 1]{DLR}). We also note that $u$
satisfies \eqref{min_eqn} if and only if 
$G:=\left\{  \left(  x,u(x)\right)\colon x\in\Omega\right\}  $ 
is a minimal hypersurface in the product space $M\times\mathbb{R}$. 

The class of equations considered here include also the usual Laplace-Beltrami equation 
($\cA(t)\equiv 1$ and $\cB(t)\equiv 0$) and, more generally, the $p$-Laplace equation 
$\diver\bigl(\abs{\nabla u}^{p-2}\nabla u\bigr)=0,\ 1<p<\infty,$ in which case 
\[
\cA(t)=t^{(p-2)/2}\text{ and } \cB(t)=\frac{p-2}{2t},
\]
and so $A_0=1$ and $B_0=(p-2)/2$.
It is well-known that the properties (A) and (B) above hold for the $p$-Laplace equation and that (weak) solutions of 
the $p$-Laplace equation have H\"older-continuous representatives, usually called $p$-harmonic functions; see 
\cite{HKM}.

The main result of this paper is the following nonsolvabity theorem whose main special case is stated separately in  
Corollary~\ref{cor:main_minimal}. We want to point out that the properties (A) and (B) are not needed in the part (a) below.
\begin{Thm}\label{thm:main0} 
There exists a $3$-dimensional Cartan-Hadamard manifold $M$ with sectional curvatures $\le -1$ such that
\begin{enumerate}
\item[(a)]
for any operator $Q$, with $\cA$ satisfying \eqref{Agrowth}, \eqref{Bgrowth}, and \eqref{A(0)}, the asymptotic Dirichlet problem for the 
equation $Q[u]=0$ is not solvable with any continuous non-constant boundary data on $M(\infty)$,
\item[(b)] there are non-constant bounded continuous $Q$-solutions on $M$ if, in addition, $Q$ satisfies the properties 
(A) and (B).
\end{enumerate} 
\end{Thm}

\begin{corollary}\label{cor:main_minimal}
There exists a $3$-dimensional Cartan-Hadamard manifold $M$ with sectional curvatures $\le -1$ such that the asymptotic 
Dirichlet problem for the minimal graph equation \eqref{min_eqn} is not solvable with any continuous non-constant boundary data, 
but there are non-constant bounded continuous solution of \eqref{min_eqn} on $M$.
\end{corollary}

The asymptotic Dirichlet problem for the Laplace-Beltrami operator has been extensively studied during the last three decades.
It was solved affirmatively by Choi \cite{choi} under assumptions that sectional curvatures satisfy $\sect\le -a^{2}<0$ and the 
so-called \define{convex conic neighborhood condition} holds. The latter means that for any pair of points 
$x,y\in M(\infty),\ x\ne y$, there exist disjoint neighborhoods $V_x, V_y\subset\bar{M}$ in the cone topology such that 
$V_x\cap M$ is convex with a $C^2$ boundary. Such appropriate convex sets were
constructed by Anderson \cite{andJDG} for manifolds of pinched sectional
curvature $-b^{2}\le \sect\le -a^{2}<0.$ Independently, Sullivan \cite{sullivan} solved the Dirichlet
problem at infinity under the same pinched curvature assumption by using probabilistic arguments. In \cite{andschoen}, 
Anderson and Schoen presented a simple and direct solution to the Dirichlet problem again in the case of pinched negative 
curvature. Important contributions to the Dirichlet problem were given by Ancona in a series of papers \cite{ancannals},  
\cite{anchyp}, \cite{ancpot}, and \cite{ancrevista}. In particular, he was able to replace the curvature lower bound  by a 
bounded geometry assumption that each ball up to a fixed radius is $L$-bi-Lipschitz equivalent to an open set in $\R^{n}$
for some fixed $L\ge 1$; see \cite{ancannals}. On the other hand, in \cite{ancrevista} Ancona constructed a 3-dimensional 
Cartan-Hadamard manifold with sectional curvatures bounded from above by $-1$ where the asymptotic Dirichlet problem is not 
solvable. Another example of a (3-dimensional) Cartan-Hadamard manifold, with sectional curvatures
$\le -1$, on which the asymptotic Dirichlet problem is not solvable was constructed by  Borb\'ely \cite{Bor}.

The Dirichlet problem at infinity has been studied also in a more general context of $p$-harmonic and $\cA$-harmonic functions
as well as for operators $Q$. In the case of the $p$-Laplace equation the Dirichlet problem at infinity was solved
in \cite{Ho} on Cartan-Hadamard manifolds of pinched negative sectional curvature by modifying the direct approach of Anderson 
and Schoen \cite{andschoen}.
In \cite{HoVa} Holopainen and V\"ah\"akangas studied the asymptotic Dirichlet problem for the $p$-Laplace equation  on a 
Cartan-Hadamard manifold $M$ under a curvature assumption
\begin{equation}\label{curv_assump}
-b\bigl(\rho(x)\bigr)^2\le \sect_x\le -a\bigl(\rho(x)\bigr)^2
\end{equation}
outside a compact set. Here $\rho(x)$ stands for the distance between $x\in M$ and a fixed point $o\in M$ and, furthermore,
$a,b\colon [0,\infty)\to [0,\infty),\ b\ge a,$ are smooth functions subject to certain growth conditions; see Theorem~\ref{CHR_thm1} and Theorem~\ref{HVkor2_RT} below for the two important special cases of functions $a$ and $b$.

Concerning the minimal graph equation \eqref{min_eqn} there has been a growing interest in entire minimal hypersurfaces in 
product spaces $M\times\R$. Indeed, in \cite{CR} Collin and Rosenberg constructed harmonic diffeomorphisms from the 
complex plane $\C$ onto the hyperbolic plane $\hyp^2$ disproving a conjecture of Schoen and Yau \cite{ScYau}. 
This result was extended by G{\'a}lvez and Rosenberg \cite{GR} to any Cartan-Hadamard surface $M$ with curvature bounded from above by a negative constant. 
The method in both papers is to construct an entire minimal surface $\Sigma=(x,u(x))\subset \hyp^2\times \R$ 
($\Sigma\subset M\times\R$, resp.) of conformal type $\C$, and thus to construct an entire unbounded solution 
$u$ to the minimal graph equation. Harmonic diffeomorphisms $\C\to \hyp^2$ ($\C\to M$, resp.) are then obtained by composing 
conformal diffeomorphisms $\C\to\Sigma$ with harmonic vertical projections $\Sigma\to\hyp^2$ ($\Sigma\to M$, resp.). 
In both papers the construction of an entire unbounded solution $u$ to the minimal graph equation is based on a 
Jenkins-Serrin type theorem \cite{JeSe} on the Dirichlet problem on unbounded ideal polygons.
Motivated by these unexpected results, by the desire to understand minimal hypersurfaces in product spaces $M\times\R$,
and by the recent research in this field (see for example, \cite{DHL}, \cite{ER}, \cite{MR}, \cite{NR}, 
\cite{RT}, \cite{RSS}, \cite{ET}, \cite{Spruck}), the authors of the current paper together with Casteras extended the results 
obtained in \cite{HoVa} for the $p$-Laplacian  to the minimal graph equation under curvature assumptions \eqref{curv_assump}. 
In fact, their results cover the equation \eqref{Q_equ}, with $\cA$ satisfying \eqref{Agrowth}, \eqref{Bgrowth}, \eqref{A(0)}, and
conditions (A) and (B). As special cases of their main theorem \cite[Theorem 1.6]{CHR} we state here the following two 
solvability results.
\begin{Thm}\cite[Theorem 1.5]{CHR}
\label{CHR_thm1} Let $M$ be a Cartan-Hadamard manifold of dimension $n\ge2$. Fix
$o\in M$ and set $\rho(\cdot)=d(o,\cdot)$, where $d$ is the Riemannian
distance in $M$. Assume that
\begin{equation*}
-\rho(x)^{2\left(  \phi-2\right)  -\varepsilon}%
\leq\sect_{x}(P)\leq-\dfrac{\phi(\phi-1)}{\rho(x)^{2}},
\end{equation*}
for some constants $\phi>1$ and $\varepsilon>0,$ where $\sect_{x}(P)$ is the
sectional curvature of a plane $P\subset T_{x}M$ and $x$ is any point in the
complement of a ball $B(o,R_{0})$. Then the asymptotic Dirichlet problem for
the minimal graph equation \eqref{min_eqn} is uniquely solvable for any
boundary data $f\in C\bigl(M(\infty)\bigr)$.
\end{Thm}
\begin{Thm}\cite[Corollary 1.7]{CHR}
\label{HVkor2_RT} Let $M$ be a Cartan-Hadamard manifold of dimension $n\ge2$.
Fix $o\in M$ and set $\rho(\cdot)=d(o,\cdot)$, where $d$ is the Riemannian
distance in $M$. Assume that
\begin{equation}\label{curv_assump_k}
-\rho(x)^{-2-\varepsilon}e^{2k\rho(x)}\le\sect_{x}%
(P)\le-k^{2}%
\end{equation}
for some constants $k>0$ and $\varepsilon>0$ and for all $x\in M\setminus
B(o,R_{0})$. Then the asymptotic Dirichlet problem for the equation
\eqref{Q_equ} is uniquely solvable for any boundary data $f\in C\bigl(M(\infty)\bigr)$.
\end{Thm}
Earlier solvability results of the asymptotic Dirichlet 
problem for the minimal graph equation were established only under hypothesis which included the condition 
$\sect_{x}(P)\leq c<0$ (see \cite{GR}, \cite{RT}). In \cite{RT} Ripoll and Telichevesky 
introduced the following \emph{strict convexity condition\/} (SC condition) that applies to equations 
\eqref{Q_equ}. A Cartan-Hadamard manifold $M$ satisfies the strict convexity condition if, for every $x\in M(\infty)$
and relatively open subset $W\subset M(\infty)$ containing $x$, there exists a $C^2$ open subset 
$\Omega\subset M$ such that $x\in\Int(M(\infty))\subset W$ and $M\setminus\Omega$ is convex.  
They proved that the asymptotic Dirichlet problem for \eqref{Q_equ} on $M$ is solvable if $\sect\le -k^2<0$
and $M$ satisfies the SC condition; see \cite[Theorem 7]{RT}. Furthermore, they showed by modifying 
Anderson's and Borb\'{e}ly's arguments that the SC condition holds on $M$ under the curvature assumption
\eqref{curv_assump_k}. Thus there exists two different kind of proofs for the result in Corollary~\ref{HVkor2_RT}.
We remark that $2$-dimensional Cartan-Hadamard manifolds $M$ with $\sect\le -k^2<0$ satisfy the SC condition
since any two points of $M(\infty)$ can be joined by a geodesic. Thus a sectional curvature upper bound $\sect\le -k^2<0$
alone is sufficient for the solvability of the asymptotic Dirichlet problem for \eqref{Q_equ} 
for $2$-dimensional Cartan-Hadamard manifolds.

All in all, it is rather surprising that asymptotic Dirichlet problems for various equations are solvable under 
essentially similar curvature assumptions.
Moreover, these solvability results have been obtained by using different kind of proofs. Indeed, Hsu \cite{Hs} solved 
the Dirichlet problem at infinity for the usual Laplace equation under quite similar curvature 
conditions than those in \ref{CHR_thm1} and \ref{HVkor2_RT} by using probabilistic arguments. In \cite{HoVa} and \cite{CHR} the asymptotic Dirichlet problem were solved by constructing barrier functions by direct computations. In \cite{Va2} V\"ah\"akangas considered so-called $\cA$-harmonic equations (of type $p\in (1,\infty)$)
\begin{equation}\label{A_harm_eqn}
-\diver\cA(\nabla u)=0
\end{equation}
and solved the asymptotic Dirichlet problem again under similar curvature assumptions. He used PDE-methods to obtain barrier functions. Above in \eqref{A_harm_eqn}, $\cA$ is subject to certain conditions; for instance $\langle\cA(V),V\rangle\approx\abs{V}^p,\ 1<p<\infty,$ 
and $\cA(\lambda V)=\lambda\abs{\lambda}^{p-2}\cA(V)$ for all $\lambda\in\R\setminus\{0\}$. Note that this class of 
equations is different from ours in the current paper, although both include the $p$-Laplace equation.
We refer to the recent paper \cite{CHR} for a more detailed discussion on the asymptotic Dirichlet problem for 
equations of type \eqref{Q_equ} and \eqref{A_harm_eqn}.

Our paper owes much to the paper \cite{Bor} by Borb\'ely. Indeed, the construction of the manifold $M$ and the idea for the proof 
of the existence of non-trivial bounded continuous solutions to $Q[u]=0$ on $M$ that can not be extended continuously to 
$M(\infty)$ are essentially due to him. On the other hand, computations and estimates for solutions to $Q[u]=0$ in 
Sections~\ref{sec:Q-varphi} and \ref{sec:construct-q} are more involved than those for the Laplacian in \cite{Bor}. For the 
details in the construction of the manifold $M$ we mainly refer to \cite{H_ns} and to the original construction \cite{Bor} by 
Borb\'ely. However, for the convenience of the reader we feel obliged to repeat quite an amount of details in the 
construction of $M$. 

\section{Main results}\label{sec:mainresults}
Our main result, Theorem~\ref{thm:main0}, follows from the  condition (a) below since it clearly implies that no
non-constant bounded continuous $Q$-solution on $M$ can have a continuous extension to $x_{0}\in M(\infty)$.
\begin{Thm}\label{thm:main1}
There exists a $3$-dimensional Cartan-Hadamard manifold $M$ with sectional curvatures $\le -1$ and a point 
$x_{0}\in M(\infty)$ such that 
\begin{enumerate}
\item[(a)] for any operator $Q$, with $\cA$ satisfying \eqref{Agrowth}, \eqref{Bgrowth}, and \eqref{A(0)},
 for all bounded continuous $Q$-solutions $u$ on $M$, and for all (cone) neighborhoods $U$ of $x_{0}$,
\[
\inf_{M}u=\inf_{U\cap M}u,\quad\sup_{M}u=\sup_{U\cap M}u,\ \text{ and}
\]
\item[(b)] there are non-constant bounded continuous $Q$-solutions on $M$ if, in addition, $Q$ satisfies the 
properties (A) and 
(B).
\end{enumerate}
\end{Thm}
The claim (b) above follows from the next result. 
\begin{Thm}\label{thm:main2}
Let $M$ and $x_{0}\in M(\infty)$ be as in \ref{thm:main1} and suppose that, in addition to \eqref{Agrowth}, \eqref{Bgrowth}, and \eqref{A(0)},
$Q$ satisfies also the properties (A) and (B). Then there exists a family of functions $u_{a,c},$ with $a\in\R$ and $c>0$, in $\bar M$ that are continuous 
$Q$-solutions on $M$, $0\le u_{a,c}\le c$, and satisfy 
\begin{enumerate}
\item[(a)] $u_{a,c}\vert M(\infty)=c \chi_{\{x_0\}}$,
\item[(b)] 
\[
\lim_{y\to x}u_{a,c}(y)=0
\]
for all $a\in\R$ and $x\in M(\infty)\setminus\{x_{0}\}$, and
\item[(c)] 
\[
\lim_{a\to-\infty}u_{a,c}(x)=c
\]
for all $x\in M$.
\end{enumerate}
\end{Thm}

The proofs of Theorem~\ref{thm:main1} and Theorem~\ref{thm:main2} are based on the following theorem.
\begin{Thm}\label{thm:main3}
There exists a $3$-dimensional Cartan-Hadamard manifold $M$ of sectional curvatures $\le -1$ and a point 
$x_{0}\in M(\infty)$ with the following properties. 
For all operators $Q$, with $\cA$ satisfying \eqref{Agrowth}, \eqref{Bgrowth}, and \eqref{A(0)},
there exist families of functions $\varphi_{a,c}$ and $\psi_{a,c}$ in $\bar M$, with $a\in\R,\ c\ge 0$, 
such that $\varphi_{a,c}$ is a continuous $Q$-subsolution on $M$, $\psi_{a,c}$ is a continuous $Q$-supersolution on $M$, 
$0\le\varphi_{a,c}\le\psi_{a,c}\le c$, and that 
\begin{enumerate}
\item[(a)] $\varphi_{a,c}\vert M(\infty)=\psi_{a,c}\vert M(\infty)=c \chi_{\{x_0\}}$,
\item[(b)]
\[
\lim_{y\to x}\psi_{a,c}(y)=0
\]
for all $a\in\R$ and $x\in M(\infty)\setminus\{x_{0}\}$, and
\item[(c)]
\[
\lim_{a\to-\infty}\varphi_{a,c}(x)=c
\]
for all $x\in M$.
\end{enumerate}
\end{Thm}
Since $0\le\varphi_{a,c}\le\psi_{a,c}\le c$, we also have 
\begin{enumerate}
\item[(b')]
\[
\lim_{y\to x}\varphi_{a,c}(y)=0
\]
for all $a\in\R$ and $x\in M(\infty)\setminus\{x_{0}\}$, and
\item[(c')]
\[
\lim_{a\to-\infty}\psi_{a,c}(x)=c
\]
for all $x\in M$.
\end{enumerate}

In order to deduce Theorem~\ref{thm:main1} and Theorem~\ref{thm:main2} from Theorem~\ref{thm:main3} we state the 
following important \define{comparison principle}, cf. \cite[Lemma 3.18]{HKM} and \cite[Lemma 3]{RT}.
We refer to \cite[Lemma 2.1]{CHR} for its short proof which is based on the fact that $t\mapsto t\cA(t^2)$ is 
strictly increasing by \eqref{Bgrowth}. 
\begin{Lem}\label{lem_comp}
If $u\in W^{1,p}(\Omega)$ is a $Q$-supersolution and $v\in W^{1,p}(\Omega)$ is a $Q$-subsolution such that
$\varphi=\min(u-v,0)\in W^{1,p}_{0}(\Omega)$, then $u\ge v$ a.e. in $\Omega$.
\end{Lem}

As a consequence, we obtain the uniqueness of $Q$-solutions with fixed (Sobolev) boundary data.
\begin{corollary}
If $u\in W^{1,p}(\Omega)$ and $v\in W^{1,p}(\Omega)$ are $Q$-solutions with $u-v\in W^{1,p}_{0}(\Omega)$, then $u=v$ a.e. in $\Omega$.
\end{corollary}

\begin{proof}[Proof of Theorem~\ref{thm:main2} assuming Theorem~\ref{thm:main3}]
Let $M,\ x_{0}\in M(\infty)$, and the families $\{\varphi_{a,c}\}$ and $\{\psi_{a,c}\}$ be as in Theorem~\ref{thm:main3}. 
Furthermore, let $\Omega_i\Subset M,\ i\in\N$, be an 
exhaustion of $M$ by $Q$-regular domains. Note that the existence of such an exhaustion is part of our assumptions on the 
operator $Q$ in Theorem~\ref{thm:main2}. For each fixed $a\in\R$ and $c>0$, let $u_{i}\in C(M)$ be the unique function that is a $Q$-solution in $\Omega_i$ 
with boundary values $\varphi_{a,c}$ and coincides with $\varphi_{a,c}$ in $\bar M\setminus\Omega_i$. By the  
comparison principle (Lemma~\ref{lem_comp}), we have $\varphi_{a,c}\le u_{i}\le\psi_{a,c}$ in $\bar M$. Thus the sequence 
$(u_{i})$ is uniformly bounded and hence, by the assumption (B) and a diagonal process,
we obtain a subsequence of $(u_{i})$ that 
converges to a function $u_{a,c}$ which is a continuous $Q$-solution in $M$, satisfies $\varphi_{a,c}\le u_{a,c}\le\psi_{a,c}$ 
in $\bar M$, and hence conditions (a)-(c) in Theorem~\ref{thm:main2}.
\end{proof}
\begin{proof}[Proof of Theorem~\ref{thm:main1} assuming Theorem~\ref{thm:main3}]
Let $M$ and $x_{0}\in M(\infty)$ be as in Theorem~\ref{thm:main3}. Condition (b) in Theorem~\ref{thm:main1}
follows from Theorem~\ref{thm:main2}.
To prove (a), suppose that $h$ is a bounded continuous $Q$-solution in $M$, $U$ is a cone 
neighborhood of $x_0$, and let
\[
b=\inf_M h,\quad\text{and}\quad B=\inf_{U\cap M} h.
\]
Then $b\le B$ and we claim that $b=B$. Write $c=B-b$ and let $\{\varphi_{a,c}\}$ and $\{\psi_{a,c}\}$, with $a\in\R$, 
be as in Theorem~\ref{thm:main3}. 
Then for each $a\in\R$ an auxiliary continuous $Q$-subsolution
\[
f_a=b+\varphi_{a,c}
\]
satisfies, for all $x\in M(\infty)\setminus\{x_{0}\}$, 
\begin{align*}
\liminf_{ \stackrel{ \mbox{\scriptsize$y\to x$} }{y\in M}}\bigl(h(y)-f_{a}(y)\bigr)
&=\liminf_{ \stackrel{ \mbox{\scriptsize$y\to x$} }{y\in M}}\bigl(h(y)-b-\varphi_{a,c}(y)\bigr)\\
&\ge \liminf_{ \stackrel{ \mbox{\scriptsize$y\to x$} }{y\in M}}\bigl(h(y)-b\bigr)-
\lim_{ \stackrel{ \mbox{\scriptsize$y\to x$} }{y\in M}}\varphi_{a,c}(y)\ge 0.
\end{align*}
Furthermore,
\begin{align*}
\liminf_{ \stackrel{ \mbox{\scriptsize$y\to x_0$} }{y\in M}}\bigl(h(y)-f_{a}(y)\bigr)
&=\liminf_{ \stackrel{ \mbox{\scriptsize$y\to x_0$} }{y\in M}}\bigl(h(y)-b-\varphi_{a,c}(y)\bigr)\\
&= \liminf_{ \stackrel{ \mbox{\scriptsize$y\to x_0$} }{y\in M}}\left(\bigl(c-\varphi_{a,c}(y)\bigr)+h(y)-B\right)\\
&\ge \liminf_{ \stackrel{ \mbox{\scriptsize$y\to x_0$} }{y\in M}}\bigl(h(y)-B\bigr)\ge 0.
\end{align*}
Hence
\begin{equation}\label{eq_hfa}
\liminf_{ \stackrel{ \mbox{\scriptsize$y\to x$} }{y\in M}}\bigl(h(y)-f_{a}(y)\bigr)\ge 0
\end{equation}
for all $x\in M(\infty)$. It follows from the comparison principle that $h\ge f_a$ in $M$ for all $a\in\R$. 
To be precise, suppose on the contrary that $h(y)<f_a(y)-\varepsilon$ for some $y\in M$ and $\varepsilon>0$.
Let $A$ be the $y$-component of the set $\{x\in M\colon h(x)<f_a(x)-\varepsilon\}$. Then $A$ is an open set with a 
compact closure $\bar A\subset M$ by \eqref{eq_hfa} and continuity of $h-f_a$. On the other hand, $h=f_a-\varepsilon$ on 
$\partial A$, and therefore $h\ge f_a-\varepsilon$ in $A$ by the comparison principle leading to a contradiction.  
Since $\lim_{a\to-\infty}\varphi_{a,c} (x)=c=B-b$ for all $x\in M$, we obtain 
\[
h(x)\ge\lim_{a\to-\infty}f_a(x)=B
\]
for all $x\in M$. Hence $b\ge B$, and so $b=B$.
\\
To complete the proof, we just apply the above to the bounded continuous $Q$-solution $-h$ and obtain
\[
\sup_M h=-\inf_M (-h)=-\inf_{U\cap M}(-h)=\sup_{U\cap M}h.
\]
\end{proof}
\begin{remark}
As is seen in the proof above, only the family $\{\varphi_{a,c}\}$ is needed in order to get the non-solvability of the 
asymptotic Dirichlet problem. 
\end{remark}
\section{Construction of $M$: First step}\label{sec:constructM} 
The construction of the Riemannian manifold $M$ is up to some minor modifications (mostly in notation) essentially 
due to Borb\'ely \cite{Bor}; see also \cite{ancrevista}, and \cite{ATU}. For the details of the construction, we refer to   \cite{H_ns}.

We start with the standard upper half space model for the hyperbolic 3-space 
\[
\hyp^{3}=\{(x^{1},x^{2},x^{3})\in\R^{3}\colon x^{3}>0\}
\] 
equipped with the hyperbolic metric $ds^{2}_{\hyp^{3}}$ of constant sectional curvature $-1$.
The sphere at infinity, $\hyp^{3}(\infty)$, can be realized as the union of the $x^{1}x^{2}$-plane and the 
''common endpoint $(x^{1},x^{2},+\infty)$`` of all vertical geodesics. Let $x_{0}=(0,0,0)\in\hyp^{3}(\infty)$ 
be a point at infinity and $L$ a unit speed geodesic terminating at $x_{0}$ ($L(+\infty)=x_{0}$) such that
$L(0)=(0,0,1)$. Thus $L$ is the positive $x^{3}$-axis. We will denote by $L$ also the image $L(\R)$. 
The Fermi coordinates $(s,r,\vartheta)$ along $L$ are defined as follows. For any point $x\in\hyp^{3}$, there exists 
a unique point $L(s)$ on $L$ closest to $x$. This determines the $s$-coordinate uniquely. The $r$-coordinate of $x$ 
is the distance $r=\dist(x,L)=d(x,L(s))$. Finally, the third Fermi coordinate $\vartheta$ of $x\in\hyp^{3}\setminus L$ 
is the angle $\vartheta\in \Sph^{1}$ obtained from the polar coordinate representation 
$x^{1}=t\cos\vartheta,\ x^{2}=t\sin\vartheta$ of $x=(x^{1},x^{2},x^{3})$. For $x=L(s)\in L$, the third Fermi 
coordinate $\vartheta$ is not defined, and we will write $x=(s,0,\ast)$. 
On $\hyp^{3}\setminus L$, the vector fields
\[
S=\frac{\partial}{\partial s},\quad
R=\frac{\partial}{\partial r},\quad \text{and}\quad
\Theta =\frac{\partial}{\partial \vartheta}
\]
form a frame, with $\{ds,dr,d\vartheta\}$ as a coframe. Furthermore, $S,\ R,$ and $\Theta$ are commuting 
as coordinate vector fields, i.e. their Lie brackets vanish:
\[
[S,R]=[S,\Theta]=[R,\Theta]=0.
\]
From now on we abbreviate $h(r)=\cosh r$ and usually write
\[
v_r'=Rv,\ v_s'=Sv,\ v_{rs}''=S(Rv),\ \text{etc.}
\]
for partial derivatives of a function $v$.

The (standard) hyperbolic metric of $\hyp^{3}$ in Fermi coordinates is given by
\[
ds^{2}_{\hyp^{3}}=dr^{2}+h^{2}(r)\,ds^{2}+\sinh^{2}r\,d\vartheta^{2}.
\]
The Riemannian manifold $M$ is then obtained from $\hyp^{3}$ by modifying the metric in $\Theta$-directions as
\begin{equation}\label{eq:defmetric}
ds^{2}_{M}=dr^{2}+h^{2}(r)\,ds^{2}+g^{2}(s,r)\,d\vartheta^{2},
\end{equation}
where $g\colon \R\times [0,+\infty[\to\R$ is a $C^{\infty}$-function which is positive in the complement of $L$, 
that is when $r>0$, 
\begin{align*}
g(s,0)&=0,\\ 
g_{r}(s,0)&:=\frac{\partial g}{\partial r}(s,0)=1,\\
\end{align*}
and whose partial derivatives of even order with respect to $r$ vanishes at $r=0$.
Thus, with respect to the Riemannian metric $ds^{2}_{M}$, we have
\begin{equation}\label{innerpr}
\langle R,S\rangle =\langle R,\Theta\rangle=\langle S,\Theta\rangle=0,\
\langle R,R\rangle =1,\ \langle\Theta,\Theta\rangle=g^2,\text{ and } \langle S,S\rangle=h^2.
\end{equation}
Above and in what follows $\langle\cdot ,\cdot\rangle$ refers to the Riemannian metric of $M$.
Furthermore, for later purposes we record the covariant derivatives of the coordinate vector fields obtained from 
\eqref{innerpr} by a direct computation:
\begin{equation}\label{covariantder}
\begin{split}
\nabla_{R}R &=0,\  \nabla_{R}S=\nabla_{S}R=\tfrac{h_r'}{h}S,\
\nabla_{R}\Theta=\nabla_{\Theta}R=\tfrac{g_r'}{g}\Theta,\ \nabla_{S}S=-hh_r' R,\\ 
\nabla_{S}\Theta&=\nabla_{\Theta}S=\tfrac{g_s'}{g}\Theta,\text{ and }
\nabla_{\Theta}\Theta=-gg_r' R-\tfrac{gg_s'}{h^2}S.
\end{split}
\end{equation}

It is crucial to note that all geodesic rays of $\hyp^3$ starting at $L$ will remain geodesic rays also in $M$, and 
therefore the sphere at infinity, $M(\infty)$, of $M$ and the cone topology of $\bar M$ can be identified with those 
of $\hyp^3$.
The Riemannian manifold $M$ will then be of sectional curvature $\le -1$ if and only if the following four 
inequalities hold:
\begin{align}
\frac{h_{rr}''}{h}&\ge 1,\label{A>1}\\
\frac{g_{rr}''}{g}&\ge 1,\label{B>1}\\
\frac{g_{ss}''}{gh^2}+\frac{g_r'h_r'}{gh}&\ge 1,\label{C>1}\\
\left(-\frac{g_{rs}''}{gh}+\frac{g_s'h_r'}{gh^2}\right)^2 &\le
\left(\frac{g_{rr}''}{g}-1\right)\left(\frac{g_{ss}''}{gh^2}+\frac{g_r'h_r'}{gh}-1\right);\label{crossterm}
\end{align}
see \cite{H_ns}.
The first condition \eqref{A>1} holds as an equality since $h(r)=\cosh r$. Thus it suffices to verify conditions \eqref{B>1} and \eqref{crossterm}.

\section{The operator $Q$ for functions $\varphi_{a,c}$}\label{sec:Q-varphi}
The family $\{\varphi_{a,c}\}$ in Theorem~\ref{thm:main3} will be constructed following the idea of Borb\'ely in \cite{Bor}. 
For $c=0$ these functions vanishes identically, therefore we assume from now on that $c>0$.
We consider a family of vector fields 
\[
X^{a}=R+q_{a}S,\quad a\in\R,
\] 
on $M\setminus L$, where, for each $a\in\R$, $q_{a}\colon M\to\R$ is a $C^{\infty}$ function depending 
only on the $r$-coordinate of a point $(s,r,\vartheta)\in M\setminus L$ and $q_{a}\vert L=0$. 
From now on we usually omit the parameter $a$ and abbreviate $X=X^{a}$, and write $q(r)=q_{a}(r)=
q_{a}(s,r,\vartheta)$. 
All integral curves of $X$ can be extended to $L$, and therefore we will talk about integral curves of $X$ 
starting at a point of $L$ even though $X$ is not defined on $X$; see \cite{H_ns} for details.
Since $X$ does not have the $\Theta$-component, the (Fermi) $\vartheta$-coordinate remains constant along integral 
curves of $X$. Furthermore, integrals curves of $X$ starting at $L(s)$ are 
rotationally symmetric around $L$; each of them is obtained from another by a suitable rotation around $L$. 
Denote by $\gamma_{a,s}$ any integral curve of $X^{a}$ starting at $L(s)$. 
Let $S_{s}^{a}$ be the surface that is obtained by rotating any $\gamma_{a,s}$ 
around $L$. 
Note also that the relation between the (Fermi) $s$-coordinate of a point $(s,r,\vartheta)\in S_{s_0}^{a}$ and 
$s_0$ is given by 
\begin{equation}\label{relation_s_0s}
s=s_0+\int_{0}^{r}q_a(t)\,dt.
\end{equation}
The functions $\varphi_{a,c}$ are constructed so that the surfaces $S_{s}^{a}$ are the level sets of 
$\varphi_{a,c}$. Thus $\varphi_{a,c}\vert S_{s}^{a}$ has a constant value $f(s)=f^{(a,c)}(s)$ 
depending only on $a,\ c$, and $s$. It is convenient to choose
\begin{equation}\label{fdef}
f(s)=f^{(a,c)}(s)=c\,\max\bigl\{0,\tanh\bigl(\delta(s-a)\bigr)\bigr\},
\end{equation}
where $\delta=\tfrac{1}{2(1+2B_{0})}$ and $B_{0}$ is the constant in \eqref{Bgrowth}. Hence $\varphi_{a,c}\vert M\setminus M_{a}=0$, where $M_{a}$ is the open set
\[
M_{a}=\bigcup_{s> a}S_s^{a}.
\]
It is worth observing that surfaces $S_s^a$ for fixed $a$ are obtained from each other by a Euclidean dilation with 
respect to $x_0$ in our upper half space model of $M$ since $q_a$ is independent of the $s$-coordinate. More precisely, 
$M_a=\{tz\colon t\in (0,1),\ z\in S_s^{a}\}$, where $tz$ stands for the (Euclidean) dilation of $z$ with respect to $x_0$. 
The functions $q=q_{a}$ will be constructed in such a way that they result in smooth functions 
$\varphi=\varphi_{a,c}$ in $M_{a}$. As in \cite{Bor} and \cite{H_ns}, we have
\begin{align}
\varphi_{s}'(s',r,\vartheta)&=f'(s),\label{Sphipositive}\\
\nabla\varphi(s',r,\vartheta)&=f'(s)\bigl(\cosh^{-2}rS-q(r)R\bigr),\label{nablaphi}\\
\abs{\nabla\varphi(s',r,\vartheta)}&=f'(s)\sqrt{\cosh^{-2}r+q^{2}(r)},\label{normnablaphi}\\
\noalign{and}
\varphi_{ss}''(s',r,\vartheta)&=f''(s),\label{SSphinegative}
\end{align}
where the (Fermi) coordinate $s'$ is related to $s$ by
\[
s'=s+\int_0^r q(t)\,dt.
\]
Note that $\abs{\nabla\varphi_{a,c}}>0$ in $M_{a}$.
Next we will compute
\[
Q[\varphi_{a,c}]=
\diver\cA(\abs{\nabla\varphi_{a,c}}^{2})\nabla\varphi_{a,c}
\]
pointwise in $M_{a}$. We start with noting that, for a $C^{2}$-function $u$ (with  $\abs{\nabla u}>0$),
\begin{align}\label{qhess}
\nonumber
\diver\cA(\abs{\nabla u}^{2})\nabla u
&=\cA(\abs{\nabla u})^{2}\Delta u + \langle\nabla\cA(\abs{\nabla u}^{2}),\nabla u\rangle \\
\nonumber
&=\cA(\abs{\nabla u})^{2}\Delta u + \cA'(\abs{\nabla u}^{2})\left\langle\nabla\langle \nabla u,\nabla u\rangle,
\nabla u\right\rangle\\
&=\cA(\abs{\nabla u})^{2}\Delta u + 2\cA'(\abs{\nabla u}^{2})\Hess u(\nabla u,\nabla u) \\
\nonumber
&=\cA(\abs{\nabla u})^{2}\left\lbrace \Delta u + 
2\cB(\abs{\nabla u}^{2})\abs{\nabla u}^{2}\Hess u\left(\tfrac{\nabla u}{\abs{\nabla u}},\tfrac{\nabla u}{\abs{\nabla u}}\right)
\right\rbrace.
\end{align}
In particular, we have 
\[
\diver\cA(\abs{\nabla\varphi}^{2})\nabla\varphi
=\cA(\abs{\nabla\varphi})^{2}\left\lbrace \Delta\varphi + 
2\cB(\abs{\nabla\varphi}^{2})\abs{\nabla\varphi}^{2}\Hess\varphi\left(\tfrac{\nabla\varphi}{\abs{\nabla\varphi}},
\tfrac{\nabla\varphi}{\abs{\nabla\varphi}}\right)
\right\rbrace
\]
for functions $\varphi=\varphi_{a,c}$ in $M_{a}$.

Following Borb\'ely, we define a $C^{\infty}$-function $\beta\colon M\to [0,\infty)$ 
(denoted by $p$ in \cite{Bor}) by
\begin{equation}\label{betadef}
\beta(s,r)=\frac{g_s'(s,r)}{g_r'(s,r)h^2(r)}.
\end{equation}
Writing $Y=\tfrac{\nabla\varphi}{\abs{\nabla\varphi}}$ and computing the Laplacian as the trace of the Hessian
in the basis $\{X,\Theta,Y\}$, cf. \cite[pp. 233-234]{Bor} and \cite{H_ns}, we obtain
\begin{align*}
\diver\cA(\abs{\nabla\varphi}^{2})\nabla\varphi
 =\cA(\abs{\nabla\varphi})^{2}\Bigl\lbrace & \frac{\Hess\varphi(X,X)}{\langle X,X\rangle} + 
\frac{\Hess\varphi(\Theta,\Theta)}{\langle \Theta,\Theta\rangle} \\
& +
\left(1+2\cB(\abs{\nabla\varphi}^{2})\abs{\nabla\varphi}^{2}\right)
\Hess\varphi(Y,Y)
\Bigr\rbrace,\\
\end{align*}
where the Hessians are obtained from \eqref{covariantder} by simple computations:
\begin{align*}
\frac{\Hess\varphi(X,X)}{\langle X,X\rangle} &=
\frac{-\varphi_s'\left(hq_r'+2h_r'q+h^2h_r'q^3\right)}{h(1+h^2q^2)},\\
\frac{\Hess\varphi(\Theta,\Theta)}{\langle \Theta,\Theta\rangle}  &=
\frac{\varphi_s'g_r'(\beta-q)}{g},\\
\noalign{and}
\Hess\varphi(Y,Y)&=\varphi_{ss}''(h^{-2}+q^2)
-\frac{\varphi_s'(q_r'q^2-h_r'h^{-3}q)}{h^{-2}+q^{2}}.
\end{align*}
Hence putting these together and simplifying we arrive at the following formula.
\begin{Lem}\label{Qphi}
The operator $Q$ for functions $\varphi=´\varphi_{a,c}$ is given in $M_{a}$ by the formula
\begin{align*}
\diver\cA(\abs{\nabla\varphi}^{2})\nabla\varphi & \\
=\frac{\cA(\abs{\nabla\varphi}^{2})\varphi_s'}{h}
&\Bigl\{
\frac{g_r'h(\beta-q)}{g} - hq_r' - h_r'q +\frac{\varphi_{ss}''(1+h^2q^2)}{\varphi_s'h}
\left[ 1+ 2\cB(\abs{\nabla\varphi}^{2})\abs{\nabla\varphi}^{2}\right]\\
&-2\cB(\abs{\nabla\varphi}^{2})\abs{\nabla\varphi}^{2}\frac{h^3q_r'q^2-h_r'q}{1+h^2q^2}\Bigr\}.
\end{align*}
\end{Lem}
\begin{remark}\label{beta-q}
It is worth noting already at this stage that, in order to have $Q[\varphi]\ge 0$, the first term above, i.e.
the one containing $\beta-q$, should be positive and dominate the others. This requirement puts strong constraints on 
functions $\beta,\ g,$ and $q$.
\end{remark}
\begin{remark}\label{q-prop}
In order to guarantee the correct boundary behavior of the functions $\varphi_{a,c}$, i.e. conditions (a) and (b') in 
Theorem~\ref{thm:main3}, it is enough to construct functions $q_a$ so that
\begin{equation}\label{qehto1}
\int_{0}^{\infty}q_a(t)\,dt=\infty
\end{equation}
for all $a\in\R$ and that 
\begin{equation}\label{qehto2}
\int_{0}^{r}q_a(t)\,dt\le b_r<\infty
\end{equation}
independently of $a\in\R$; see \cite[Lemma 5.1]{H_ns}.  
\end{remark}

\section{Construction of $M$: Final step}\label{sec:construct-g} 
In this section we briefly describe the construction of the function $g$ in \eqref{eq:defmetric} and hence the Riemannian metric 
of $M$. The function $g$ will be of the form
\begin{equation}\label{eq:g-def}
g(s,r)=\frac 12\sinh\bigl(\sinh 2\varrho(s,r)\bigr),
\end{equation}
where $\varrho$ is a $C^{\infty}$-function, with $\varrho(s,r)=r$ for $0\le r\le 3$ and 
$\varrho(s,r)\ge r$ for all $r\ge 0$.
By \eqref{betadef}, $g$ and $\varrho$ both satisfy the partial differential equation
\begin{align}
g_s'&=\beta h^2 g_r',\label{gPDE}\quad\text{and}\\
\varrho_s'&=\beta h^2 \varrho_r'.\label{rhoPDE}
\end{align}
Note that $\beta$ is independent of the (Fermi) coordinate $\vartheta$ and $\beta(s,r)=0$ for 
$0\le r\le 3$ by \eqref{eq:g-def}. 
Since $\nabla\varrho=\varrho_r'(\beta S+R)$, we have 
$\nabla\varrho\perp (\beta h^{2}R-S)$, and therefore $\varrho$ (and hence $g$) is constant along any integral curve of the vector field
\[
Z=\beta h^{2}R-S.
\] 
Now the idea is to construct an unbounded domain $\Omega\subset M$ of the form
\begin{equation}\label{omegadef}
\Omega=\{(s,r,\vartheta)\in M\colon r<3 \}\cup \{(s,r,\vartheta)\in M\colon s< -\ell(r) \}
\end{equation}
such that all integral curves of $Z$ will enter at $\Omega$, and then construct
$\beta$ so that it vanishes identically in $\Omega$, and finally fix the ''initial condition``
\begin{equation}\label{varrhodef}
\varrho(s,r)=r
\end{equation}
for all $(s,r,\vartheta)\in\Omega$. Note that $(s,r,\vartheta)\in\Omega$ for all $s\le s'$ if 
$(s',r,\vartheta)\in\Omega$. Consequently, once an integral curve of $Z$ enters at $\Omega$, it will then stay in 
$\Omega$ forever. The function $\ell$ that appears in \eqref{omegadef} is closely related to $\beta$.
Then $g$, and hence the Riemannian structure of $M$, will be completely determined by constructing the functions 
$\beta$ and $\ell$. 

While constructing $\beta$ we have to keep in mind Remark~\ref{beta-q} and \eqref{qehto1}. This leads to the first requirement that
\begin{equation}\label{betaehto1}
\int_0^{\infty}\beta(s,r)\,dr=\infty
\end{equation}
for all $s\in\R$. For the construction of $g$ we require that 
\begin{equation}\label{betaehto0}
\int_{r_0}^{\infty}\dfrac{dr}{\beta(s,r)\cosh^2 r}=\infty
\end{equation}
for all $r_0>3$ and $s\in\R$. To obtain the curvature conditions \eqref{B>1} and \eqref{crossterm} we will require 
that 
\begin{equation}\label{betaehto2}
0\le\beta\le \frac{1}{1000},\quad\abs{\beta_r'}\le\frac{1}{1000},\quad
0\le \beta_s'\le\frac{1}{1000},\quad\beta\beta_r'h^3\le\frac{h_r'}{1000},
\end{equation}
and that $\beta h^2$ is a convex non-decreasing function in the variable $r$, that is
\begin{equation}\label{betaehto3}
(\beta h^2)_r'\ge 0\quad\text{and}\quad (\beta h^2)_{rr}''\ge 0.
\end{equation}
The function $\beta\colon M\to [0,\infty)$ will be of the form
\[
\beta(s,r)=\xi(s+\ell(r))\beta_0(r),
\]      
with smooth functions $\xi,\ \ell,$ and $\beta_0$ to be described next.
The function $\xi\colon\R\to [0,1]$ is smooth and non-decreasing such that
$\xi\vert (-\infty,0]=0,\ \xi\vert [4,\infty)=1,$ and that $\xi',\abs{\xi''}<1/2,$ and $\xi''+\xi>0$ on $(0,4)$.
The smooth function $\beta_0\colon [0,\infty)\to [0,\infty)$ is constructed such that on $[0,r_1]$ it is a 
smooth non-decreasing function that vanishes identically on $[0,3],$ takes the constant (positive) value 
$\beta_0(5)$ on $[5,r_1],$ and it is a positive slowly increasing function on the interval $(3,5]$ so that 
\eqref{betaehto2} and \eqref{betaehto3} hold. Here $r_1$ is large enough such that 
$\beta_0(r_1)\cosh^2 r_1=\beta_0(5)\cosh^2 r_1>1$. Furthermore, $\beta_0$ is non-increasing on $[r_1,\infty)$, with 
$\lim_{r\to\infty}\beta_0(r)=0$, whereas $\beta_0h^2$ is an 
increasing strictly convex function. Finally, $\beta_0$ satisfies
\[
\int_0^{\infty}\beta_0(r)\,dr=\infty,\quad
\int_{r_0}^{\infty}\dfrac{dr}{\beta_0(r)\cosh^2 r}=\infty
\]
for all $r_0>3$, and
\begin{equation}\label{betah2''}
(\beta_0h^2)_{rr}''>\dfrac{\varepsilon}{\beta_0h^2},
\end{equation} 
where $0<\varepsilon<1/4$ is small enough depending on the choice of $\beta_0\vert [0,5]$. We refer to \cite{Bor} for a detailed 
construction of $\beta_0$ (denoted there by $p_0$); see also \cite{H_ns}.

The smooth function $\ell\colon [0,\infty)\to [0,\infty)$ is constructed so that $\ell(r)=0$ for $r\in [0,3]$ and
\begin{equation}\label{ell'}
\ell'=\dfrac{\varepsilon}{\beta_0h^2 }
\end{equation}
on the interval $[5,\infty)$, with the same $\varepsilon$ as in \eqref{betah2''}. Finally, the two pieces are connected 
smoothly such that 
\[
\ell''\ge\dfrac{-\ell'(\beta_0h^2)'_r}{\beta_0h^2}\quad\text{and}\quad
0\le\ell'\le\dfrac{\varepsilon}{\beta_0h^2}
\]
for all $r>0$. Then $\ell(r)\to\infty$ as $r\to\infty$ and $\beta(s,r)=\xi(s+\ell(r))\beta_0(r)$ satisfies the conditions 
\eqref{betaehto1}--\eqref{betaehto3}; see \cite{Bor} for the details.

Next we complete the construction of $g$. Recall from \eqref{omegadef} and \eqref{varrhodef} that
\[
\Omega=\{(s,r,\vartheta)\in M\colon r<3\}\cup\{s,r,\vartheta)\in M\colon s<-\ell(r)\}
\]
and hence $\beta\equiv 0$ and $g(s,r)=\tfrac 12 \sinh(\sinh 2r)$ in $\bar\Omega$ and $\beta>0$ in 
$M\setminus\bar\Omega$. Notice that integral curves of $W=R-\ell'S$ starting at points in 
$\partial\Omega\cap\{(s,r,\vartheta)\in M\colon r>3\}$ will stay in $\partial\Omega$. Since
\[
\dfrac{1}{\beta h^2}-\ell'\ge \dfrac{1}{\beta_0 h^2}-\dfrac{\varepsilon}{\beta_0 h^2}=\dfrac{1-\varepsilon}{\beta_0 h^2}
\]
we conclude from \eqref{betaehto0} that all integral curves of $Z=\beta h^2R-S$ starting at points in
$M\setminus\bar\Omega$ will enter at $\Omega$ and stay in there,  see \cite[p. 229]{Bor}. 
As observed earlier, $\varrho$ and $g$ are constant along any integral curves of $Z$. 
This completes the construction of $g$ and the Riemannian metric of $M$. We refer to \cite{H_ns} for the proof of  
the curvature conditions \eqref{B>1} and \eqref{crossterm}. 

We finish this section by collecting further properties of $g$ that will be used in Section~\ref{sec:construct-q}.
Recall from \eqref{eq:g-def} and \eqref{rhoPDE} that
\[
g(s,r)=\frac 12\sinh\bigl(\sinh 2\varrho(s,r)\bigr),
\]
where $\varrho$ is a $C^{\infty}$-function, with $\varrho(s,r)=r$ for $0\le r\le 3$, that 
satisfies the partial differential equation
\[
\varrho_s'=\beta h^2 \varrho_r'.
\]
Hence we may apply the proof of \cite[Lemma 2.2]{Bor} to the function $\varrho$. 
Since $\varrho_r'(s,r)\equiv 1$ in $\Omega$, 
we get  
\begin{equation}\label{rho'}
\varrho_r'\ge 1
\end{equation}
and
\begin{equation}\label{rho}
\varrho(s,r)\ge r
\end{equation}
in $M$. 
Furthermore, 
\begin{equation}\label{g_r_est}
g_r'=\varrho_r'\cosh(2\varrho)\cosh(\sinh 2\varrho)
\end{equation}
and
\begin{equation}\label{g_r/g_est}
\frac{g_r'}{g}=2\varrho_r'\coth(\sinh 2\varrho)\cosh 2\varrho.
\end{equation}

\section{Construction of $Q$-subsolutions $\varphi_{a,c}$}\label{sec:construct-q}   
In this section we construct the functions $q_a\colon [0,\infty)\to\R,\ a\in\R$, so that the resulting functions
$\varphi_{a,c}$ satisfy the conditions in Theorem~\ref{thm:main3}. For each fixed $a\in\R$, we first define $q=q_a$ 
piecewise on intervals $[0,T_0],\ [T_0,T_1],\ [T_1,T_2],\ [T_2,T_3]$, and $[T_3,\infty)$, where $T_0,\ldots,T_3$ 
depend only on $a$ and $B_{0}$, and then finally smooth out $q$ in neighborhoods of $T_i,\ i=0,1,2,3$. We denote both the piecewisely 
constructed functions and the final smooth functions by the same symbol $q$. 

Recall from \eqref{fdef} that 
\[
f(s)=f^{(a,c)}(s)=c\,\max\bigl\{0,\tanh\bigl(\delta(s-a)\bigr)\bigr\}.
\]
with $\delta=\tfrac{1}{2(1+2B_{0})}$. Furthermore, by \eqref{Sphipositive} and \eqref{SSphinegative}, we have
\[
\varphi_s'(s',r,\vartheta)=f'(s) > 0
\]
and
\[
\varphi_{ss}''(s',r,\vartheta)=f''(s) < 0
\]
for $s>a$, where 
\[
s'=s+\int_{0}^{r}q(t)\,dt.
\]
Hence
\[
\frac{\varphi_{ss}''(1+h^2q^2)}{\varphi_s'h}
\left[ 1+ 2\cB(\abs{\nabla\varphi}^{2})\abs{\nabla\varphi}^{2}\right]
\ge -\frac{1+h^2q^2}{h}
\]
in $M_a$. We conclude from Lemma~\ref{Qphi} that $Q[\varphi]>0$ in $M_a$ if 
\begin{equation}\label{Qsubphi}
\frac{g_r'h(\beta-q)}{g} - hq_r' - h_r'q  -\frac{1+h^2q^2}{h}
-2\bar{B}_0\frac{\abs{h^3q_r'q^2-h_r'q}}{1+h^2q^2}>0
\end{equation}
in $M_a$, where $\bar{B}_0=\max(B_{0},1/2)$.

It is straightforward to check that integral curves of vector fields $R-\tanh r\, S,\ r>0,$ are horizontal (Euclidean) 
lines,
i.e. the $x^3$-coordinate remains constant along an integral curve. Hence we define $q(r)=q_a(r)=-\tanh r$ for 
$r\in [0,T_0]$, where $T_0\ge 1$ will be chosen later. Then the surfaces $S_s^{a}$ coincide with horizontal Euclidean 
planes $x^3\equiv e^{-s}$ 
near $L$. Consequently, the functions $\varphi_{a}$ are smooth in $M_a$. We notice that 
\[
q_r'=-\cosh^{-2}r,\quad
1+h^2q^2=\cosh^2 r,\quad
-hq_r'-h_r'q=\cosh r,\\
\]
and
\[
\frac{h^3q_r'q^2 -h_r'q}{1+h^2q^2}=0.
\]
Furthermore, since $\beta\ge 0$, we get from \eqref{g_r/g_est} and \eqref{Qsubphi} that 
$Q[\varphi]>0$ in $M_a\cap\{(s',r,\vartheta)\colon 0<r< T_0\}$ because there
\begin{align}\label{q1}
&\frac{g_r'h(\beta-q)}{g} - hq_r' - h_r'q  -\frac{1+h^2q^2}{h}
-2\bar{B}_0\frac{\abs{h^3q_r'q^2-h_r'q}}{1+h^2q^2}\\
\nonumber
\ge &
\coth(\sinh 2\varrho)\sinh(2\varrho) \coth(2\varrho)2\varrho\varrho_r'\frac{\sinh r}{\varrho}
\ge\  \varrho_r'\frac{\sinh r}{\varrho}>0.
\end{align}
Since $\varrho=r$ and hence $\varrho_r'=1$ for $0\le r \le 3$, we have $\varrho_r'\varrho^{-1}\sinh r\to 1$ as 
$r\to 0$. 

For $r\in [T_0,T_1]$, we define
\[
q(r)=q_a(r)=\frac{-\cosh T_0 \sinh r}{\cosh^2 r}.
\]
Then 
\begin{align*}
& q_r'(r)=\cosh T_0(\sinh^2 r-1)\cosh^{-3}r,\quad
1+h^2q^2=1+\cosh^2 T_0\tanh^2 r,\\
&-hq_r'-h_r'q=\cosh T_0\cosh^{-2}r,
\end{align*}
and
\begin{align*}
\frac{\abs{h^3q_r'q^2 -h_r'q}}{1+h^2q^2} &=\frac{\cosh^3T_0\tanh^2r(\tanh^2r-\cosh^{-2}r)+\cosh T_0\tanh^2r}
{1+\cosh^2 T_0\tanh^2 r}\\
&\le \cosh T_0.
\end{align*}
Again since $\beta\ge 0$, we may estimate the left hand side of \eqref{Qsubphi} from below to obtain
\begin{align}\label{q2}
\nonumber
\frac{g_r'h(\beta-q)}{g} & - hq_r' - h_r'q  -\frac{1+h^2q^2}{h}
-2\bar{B}_0\frac{\abs{h^3q_r'q^2-h_r'q}}{1+h^2q^2}\\
 \ge\ \cosh T_0 &\Bigl(2\varrho_r'\coth(\sinh 2\varrho)\cosh(2\varrho)\tanh r + \cosh^{-2}r\\
\nonumber
&-\frac{1+\cosh^2 T_0\tanh^2 r}{\cosh T_0\cosh r} - 2\bar{B}_0 \Bigr)\\
\nonumber
> \cosh T_0 &\Bigl(\cosh 2r +\cosh^{-2}r-\cosh^{-2}T_0 -\tanh^2 r-2\bar{B}_0\Bigr)
>0
\nonumber
\end{align}
in $M_a\cap\{(s',r,\vartheta)\colon T_0 < r < T_1\}$, where $T_0=T_0(B_{0})\ge 1$ is large enough. 

For $r\in [T_1,T_2]$, we let $q=q_a$ be a $C^{\infty}$ continuation of $q\vert[0,T_1]$ such that 
\begin{equation*}
\frac{-\cosh T_0 \sinh r}{\cosh^2 r} \le  q \le 0
\end{equation*}
and
\begin{equation*}
0<\left(\frac{-\cosh T_0 \sinh r}{\cosh^2 r}\right)^{\prime}_r  < q_r' <
\frac{\cosh T_0}{\cosh r}.
\end{equation*}
Thus
\begin{align*}
-hq_r' - h_r' &\ge -\cosh T_0,\\
1\le 1+h^2q^2 &\le 1+\cosh^2 T_0\tanh^2 r,\\
-\frac{1+h^2q^2}{h} &\ge -\frac{1+\cosh^2 T_0\tanh^2 r}{\cosh r},\\
\noalign{and}
-2\bar{B}_0\frac{\abs{h^3q_r'q^2 -h_r'q}}{1+h^2q^2} &= -2\bar{B}_0\frac{h^3q_r'q^2 -h_r'q}{1+h^2q^2} \\
 &\ge -2\bar{B}_0(h^3q_r'q^2 -h_r'q)\\
 &\ge -2\bar{B}_0\cosh T_0\tanh^{2}r(\cosh^2 T_0+1).
\end{align*}
We choose $T_1=T_1(a,T_0)>T_0$ large enough so that $s'+\ell(r)\ge 4$ for all 
$s'\ge a-\log\cosh T_0 -1$ and $r\ge T_1$, 
which then implies that for all $s\ge a$ and $r\in [T_1,T_2]$ the point $(s',r,\vartheta)$ on any 
integral curve $\gamma_{a,s}$ of $X^{a},$ with 
\[
s'=s+\int_{0}^{r}q(t)\,dt\ge a-\log\cosh T_0 -1,
\] 
lies in the set where $\beta(s',r)=\beta_0(r)$. Furthermore, we also require that $T_1$ is so large that 
$\beta_0(r)\cosh^2 r\ge 1$ 
for all $r\ge T_1.$ Then in $M_a\cap\{(s',r,\vartheta)\colon T_1< r< T_2\}$, with $T_1$ large enough, we have
\begin{align}\label{q3}
\nonumber
\frac{g_r'h(\beta-q)}{g} &- hq_r' - h_r'q  -\frac{1+h^2q^2}{h}
-2\bar{B}_0\frac{\abs{h^3q_r'q^2-h_r'q}}{1+h^2q^2}\\
\nonumber
\ge\ &
2(\beta_0-q)\varrho_r'\coth(\sinh 2\varrho)\cosh(2\varrho)\cosh r -\cosh T_0\\
& -\frac{1+\cosh^2 T_0 \tanh^2 r}{\cosh r}
-2\bar{B}_0\cosh T_0\tanh^2 r(\cosh^2 T_0+1)\\
\nonumber
\ge\ &
2\cosh 2r \cosh^{-1}r -\cosh T_0 -\frac{1+\cosh^2 T_0 \tanh^2 r}{\cosh r}\\
\nonumber
 & 
-2\bar{B}_0\cosh T_0\tanh^2 r(\cosh^2 T_0+1) > 0.
\nonumber
\end{align}
Here we used estimates $\beta_0-q\ge \cosh^{-2}r$ and $2\varrho_r'\coth(\sinh 2\varrho)\cosh(2\varrho)\ge 2\cosh 2r$ for 
$r\ge T_1$.
The upper interval bound $T_2$ is determined by $q(T_2)=0$. Such $T_2$ exists since $q$ grows strictly faster than 
\[
r\mapsto \frac{-\cosh T_0 \sinh r}{\cosh^2 r}
\]
which tends to zero as $r\to\infty$. 
Since
\[
\int_{t}^{\infty}\beta_0(r)\,dr=\infty\quad
\text{and}\quad
\int_{t}^{\infty}\frac{dr}{\beta_0(r)\cosh^2 r}=\infty
\]
for every $t>3$, $\beta_0(r)-1/\cosh r$ changes its sign infinitely often, in particular, there are arbitrary large 
values of $r$, with $\beta_0(r)-1/\cosh r=0$. We let $T_3>T_2$ be a zero of $\beta_0-1/\cosh$ specified later.
For $r\in [T_2,T_3]$ we let $q(r)=0$. Then 
\begin{align}\label{q4}
\nonumber
\frac{g_r'h(\beta-q)}{g} &- hq_r' - h_r'q  -\frac{1+h^2q^2}{h}
-2\bar{B}_0\frac{\abs{h^3q_r'q^2-h_r'q}}{1+h^2q^2}\\
=\ & \frac{g_r'h\beta_0}{g}-\frac{1}{h}
\ge \frac{2\cosh 2r -1}{\cosh r}>0
\end{align}
in $M_a\cap\{(s',r,\vartheta)\colon T_2< r< T_3\}.$

For $r\ge T_3$ we define $q(r)=\beta_0(r)-1/\cosh r$. Then   
\begin{align*}
\beta-q &=\beta_0-q=1/\cosh r,\\
-hq_r'-h_r'q&=-\beta_0'(r)\cosh r -\beta_0(r)\sinh r,\\
-\frac{1+h^2q^2}{h}&=-\beta_0^2(r) \cosh r + 2\beta_0(r)-2\cosh^{-1}r,\\
\noalign{and}
-2\bar{B}_0\frac{\abs{h^3q_r'q^2-h_r'q}}{1+h^2q^2} &\ge
-2\bar{B}_0\frac{\abs{h^3q_r'q^2}+\abs{h_r'q}}{1+h^2q^2}\\ &\ge 
-2\bar{B}_0\left(\abs{\beta_0'(r)}\cosh r +\beta_0(r)\sinh r +2\tanh r\right).
\end{align*}
Hence
\begin{align}\label{q5}
\nonumber
\frac{g_r'h(\beta-q)}{g} &- hq_r' - h_r'q  -\frac{1+h^2q^2}{h}
-2\bar{B}_0\frac{\abs{h^3q_r'q^2-h_r'q}}{1+h^2q^2}\\
\nonumber
\ge\ &
2\cosh 2r-\beta_0'(r)\cosh r-\beta_0(r)\sinh r\\
& -\beta_0^2(r)\cosh r+2\beta_0(r)-2\cosh^{-1} r\\
\nonumber
&-2\bar{B}_0\left(\abs{\beta_0'(r)}\cosh r +\beta_0(r)\sinh r +2\tanh r\right) >0
\nonumber
\end{align}
in $M_a\cap\{(s',r,\vartheta)\colon r>T_3\}$ if $T_3$ is large enough. Finally, since the estimates in 
\eqref{q1}-\eqref{q5} involve $q$ and 
$q_r'$ but not higher order derivatives of $q$, it is clear that $q$ can be smoothen out in neighborhoods of $T_i$ such that 
\eqref{Qsubphi} holds in $M_a$. Hence $\varphi_{a,c}$ is a positive $Q$-subsolution in $M_a$ and continuous in $M$, 
with $\varphi_{a,c}=0$ in $M\setminus M_a$. Next we use the divergence theorem to show that $\varphi_{a,c}$ is a 
$Q$-subsolution in whole $M$. To this end, let $\eta\in C^{\infty}_{0}(M)$ be an arbitrary non-negative test function 
and let $U\Subset M$ be an open set such that $\spt\eta\subset U$ and that $\partial(M_a\cap U)$ is smooth. 
Since $\varphi_{a,c}=0$ in $M\setminus M_a$, $\eta=0$ in $M\setminus U$, and 
$\eta\diver\cA(\abs{\nabla\varphi_{a,c}}^{2})\nabla\varphi_{a,c}\ge 0$ pointwise in $M_a\cap U$, we obtain from the 
divergence theorem that
\begin{align*}
&\int_{M}\bigl\langle \cA\bigl(\abs{\nabla\varphi_{a,c}}^{2}\bigr)\nabla\varphi_{a,c},\nabla\eta\bigr\rangle dm 
=
\int_{M_a\cap U}\bigl\langle \cA\bigl(\abs{\nabla\varphi_{a,c}}^{2}\bigr)\nabla\varphi_{a,c},\nabla\eta\bigr\rangle dm \\
= &-
\int_{M_a\cap U}\eta\diver\cA(\abs{\nabla\varphi_{a,c}}^{2})\nabla\varphi_{a,c}\, dm
+\int_{\partial(M_a\cap U)}
\bigl\langle \eta\cA\bigl(\abs{\nabla\varphi_{a,c}}^{2}\bigr)\nabla\varphi_{a,c},\nu\bigr\rangle d\sigma\\
\le 
&\int_{\partial(M_a\cap U)}
\bigl\langle \cA\bigl(\abs{\nabla\varphi_{a,c}}^{2}\bigr)\nabla\varphi_{a,c},\eta\nu\bigr\rangle d\sigma,
\end{align*}
where $d\sigma$ is the (Riemannian) surface measure and $\nu$ is the unit outer normal vector field on 
$\partial(M_a\cap U)$. Furthermore,
\[
\eta\nu = -\tfrac{\eta\nabla\varphi_{a,c}}{\abs{\nabla\varphi_{a,c}}}
\] 
on $\partial(M_a\cap U)$, and therefore
\[
\int_{\partial(M_a\cap U)}
\bigl\langle \cA\bigl(\abs{\nabla\varphi_{a,c}}^{2}\bigr)\nabla\varphi_{a,c},\eta\nu\bigr\rangle d\sigma
= -\int_{\partial(M_a\cap U)}\eta\cA\bigl(\abs{\nabla\varphi_{a,c}}^{2}\bigr)\abs{\nabla\varphi_{a,c}}d\sigma
\le 0.
\]
We conclude that $\varphi_{a,c}$ is a $Q$-subsolution in the whole $M$.
Finally, 
\[
\int_{0}^{\infty}q_a(t)\, dt=\infty
\]
for all $a\in\R$ since $q_a(t)=\beta_0(t)-1/\cosh t$ for $t\ge T_3$,
\[
\int_{T_3}^{\infty}\beta_0(t)\, dt=\infty,
\]
and
\[
\int_{T_3}^{\infty}\frac{dt}{\cosh t}\le\int_{0}^{\infty}\frac{dt}{\cosh t}=\pi/2.
\]
Furthermore,
\[
\int_{0}^{r}q_a(t)\, dt\le \int_{0}^{r}\beta_0(t)\, dt +\int_{0}^{r}\frac{dt}{\cosh t}=: b_r<\infty
\]
independently of $a\in\R$. Hence the family $\{\varphi_{a,c}\}$ satisfies conditions (a), (b'), and (c) in 
Theorem~\ref{thm:main3}.

\section{Construction of $Q$-supersolutions $\psi_{a,c}$}\label{sec:construct-psi}
The construction of the family of continuous $Q$-supersolutions $\psi_{a,c},\ a\in\R,\ c>0,$ 
is similar to that in \cite{Bor} and \cite{H_ns}. It is based on the following theorem from 
e.g.~\cite[Theorem 4.3]{choi}:
\begin{Thm}\label{thm:choi4.3}
Let $N$ be an $n$-dimensional Cartan-Hadamard manifold with sectional curvatures $\le -1$. Let $\Omega\subset N$ be a 
domain with $C^{\infty}$-smooth boundary such that $\bar\Omega$ is convex. Then the distance function 
$\rho\colon N\setminus\bar\Omega\to (0,\infty),$ 
\[
\rho(x)=\dist(x,\bar\Omega),
\] 
is $C^{\infty}$ and 
\begin{equation}\label{laprho}
\Delta \rho\ge (n-1)\tanh\rho
\end{equation}
in $N\setminus\bar\Omega.$
 \end{Thm}
Suppose then that $\bar\Omega\subset N$ is a convex set and $\rho=\dist(\cdot,\bar\Omega)$ is a distance function as in 
Theorem~\ref{thm:choi4.3}. Define a continuous function $v\colon N\to [0,c)$ by setting $v=0$ in $\bar\Omega$ and 
$v(x)=c\,\tanh\bigl(\delta\rho(x)\bigr)$ for $x\in N\setminus\bar\Omega$, where $c>0$ and $\delta=\delta(B_{0})$ is a positive 
constant depending only on the constant $B_{0}$ in \eqref{Bgrowth}. 
Then in $N\setminus\bar\Omega$ we have
\[
\nabla v=c\delta\cosh^{-2}(\delta\rho)\nabla\rho
\]
and
\[
\abs{\nabla v}=c\delta\cosh^{-2}(\delta\rho).
\]
To compute $Q[v]$, we first observe that 
\begin{align*}
\Hess v\left(\tfrac{\nabla v}{\abs{\nabla v}},\tfrac{\nabla v}{\abs{\nabla v}}\right)
&= \Hess v(\nabla\rho,\nabla\rho)
= \nabla\rho\langle\nabla v,\nabla\rho\rangle - (\nabla_{\nabla\rho}\nabla\rho)v\\
&= \nabla\rho\bigl(c\delta\cosh^{-2}(\delta\rho)\bigr)
= \frac{-2c\delta^2\tanh(\delta\rho)}{\cosh^2 (\delta\rho)}
\end{align*}
and that
\[
\Delta v=\diver\bigl(c\delta\cosh^{-2}(\delta\rho)\nabla\rho\bigr)
=c\delta\cosh^{-2}(\delta\rho)\bigl(\Delta\rho - 2\delta\tanh(\delta\rho)\bigr).
\]
Hence by  \eqref{Bgrowth}, \eqref{qhess}, and \eqref{laprho} we have
\begin{align*}
Q[v]&=\diver\cA(\abs{\nabla v}^{2})\nabla v\\
&=\cA(\abs{\nabla v}^{2})\left\lbrace \Delta v +
2\cB(\abs{\nabla v}^{2})\abs{\nabla v}^{2}\Hess v\left(\tfrac{\nabla v}{\abs{\nabla v}},\tfrac{\nabla v}{\abs{\nabla v}}\right)
\right\rbrace\\
&\ge \frac{c\delta\cA(\abs{\nabla v}^{2})}{\cosh^{2}(\delta\rho)}\bigl((n-1)\tanh\rho-2\delta(1+2B_{0})\tanh(\delta\rho)\bigr).
\end{align*}
Choosing $\delta=\min(1,\tfrac{1}{2(1+2B_{0})})$ yields
\[
\diver\cA(\abs{\nabla v}^{2})\nabla v \ge 0
\]
in $N\setminus\bar\Omega$. Hence the function $\psi=c-v$ is a continuous positive function in $N$, a $Q$-supersolution 
in $N\setminus\bar\Omega$, $\psi=c$ in $\bar\Omega$, and $\psi(x)\to 0$ as $\dist(x,\bar\Omega)\to\infty$. By a similar argument 
based on the divergence theorem as in the previous section, we conclude that $\psi$ is, in fact, a $Q$-supersolution in whole 
$N$.

Thus to construct the family $\{\psi_{a,c}\},\ a\in\R,\ c>0,$ it is enough to find appropriate convex subsets of $M$. 
This is done in \cite{Bor} as follows.
Denote by $\alpha_{a}$ any integral curve of $-\nabla_{\Theta}\Theta=gg_r'(R+\beta S)$ starting at 
$L(a)$.
Furthermore, denote by $P_a$ the surface obtained by rotating $\alpha_a$ around $L$ and let $V_a$ be the 
component of $M\setminus P_a$ containing points $L(s)$, with $s>a$. Observe that $P_a$ is also obtained by rotating integral 
curves of $R+\beta S$ starting at $L(a)$ around $L$.
It is proven in \cite[p. 235]{Bor} that $\bar V_a$ is convex for every $a\in\R$. 
Next we observe that, for each fixed $a\in\R,$ the set $M_a=\{x\in M\colon \varphi_a(x)>0\}$ is 
contained in $\bar V_{a-b}$ for some $b=b(a,B_{0})$. This is seen by comparing the (Fermi) $s$-coordinates
of points $(s'',r,\vartheta)$ and $(s',r,\vartheta)$ on integral curves $\alpha_{a-b}$ and 
$\gamma_{a,s},\ s\ge a,$ respectively. More precisely, $s'\ge s''$ for all such points 
$(s'',r,\vartheta)$ and $(s',r,\vartheta)$ if $b=b(a,B_{0})$ is large enough since 
$\beta_0(r)-q_a(r)=1/\cosh r$ for $r\ge T_3=T_3(a,B_{0})$ and 
$\int_0^{\infty}1/\cosh r\, dr=\pi/2<\infty.$ Finally, for each $a\in\R$ and $c>0$, let 
$\psi_{a,c}=c-v_{a,c}$, where $v_{a,c}=c\tanh(\delta\rho_a)$, where $\rho_a=\dist(\cdot,\bar V_{a-b})$ and 
$\delta=\tfrac{1}{2(1+2B_{0})}$. Then, by the discussion above, $\psi_{a,c}$ is a continuous 
positive $Q$-supersolution in $M$, $0\le\varphi_{a,c}\le\psi_{a,c}\le c$, $\psi_{a,c}=c$ in $\bar V_{a-b}$,
and $\lim_{y\to x}\psi_{a,c}(y)=0$ for all $y\in M(\infty)\setminus\{x_0\}$.

In conclusion, the families $\{\varphi_{a,c}\}$ and $\{\psi_{a,c}\}$ satisfy the conditions in Theorem~\ref{thm:main3}, and thus 
Theorems~\ref{thm:main0}, \ref{thm:main1}, \ref{thm:main2}, and \ref{thm:main3} are proven.

\end{document}